\documentclass[12pt]{amsart}

\usepackage{amssymb,latexsym}
\usepackage{amsmath,amsfonts,amsthm,amscd,amsxtra,enumerate}
\usepackage{mathtools}
\usepackage[all]{xy}
\usepackage{hyperref}
\usepackage{amsmath,latexsym,amssymb, enumerate, amsthm, epsfig }
\usepackage{bookmark}
\usepackage[margin=1in]{geometry}
\usepackage{epsfig}
\usepackage{graphicx}
\usepackage{enumerate}

\newtheorem{con}{Conjecture}[section]
\newtheorem{defn}[con]{Definition}

\newtheorem{thm}[con]{Theorem}
\newtheorem{lemma}[con]{Lemma}
\newtheorem{remark}[con]{Remark}

\newtheorem{proposition}[con]{Proposition}

\newtheorem*{con*}{Conjecture}
\AtBeginDocument{}
\begin{document}
\title{ Some results on the Ryser design conjecture}
\author[Tushar D. Parulekar]{Tushar D. Parulekar}
\address{Department of Mathematics, I.I.T. Bombay, Powai, Mumbai 400076.}
\email{tushar.p@math.iitb.ac.in}
\author[Sharad S. Sane]{Sharad S. Sane}
\address{Chennai Mathematical Institute, SIPCOT IT Park, Siruseri, Chennai 603103}
\email{ssane@cmi.ac.in}
\date{\today}

\subjclass[2010]{05B05;    51E05;     62K10}

\keywords{Ryser design}
\begin{abstract}
	A Ryser design $\mathcal{D}$ on $v$ points is a collection of $v$ proper subsets (called blocks) of a point-set with $v$ points such that every two blocks intersect each other in $\lambda$ points (and $\lambda < v$ is a fixed number) and there are at least two block sizes. A design $\mathcal{D}$ is called a symmetric design, if every point of $\mathcal{D}$ has the same replication number (or equivalently, all the blocks have the same size) and every two blocks intersect each other in $\lambda$ points. The only known construction of a Ryser design is via block complementation of a symmetric design. Such a Ryser design is called a Ryser design of  Type-1. This is the ground for the Ryser-Woodall conjecture: ``every Ryser design is of Type-1". This long standing conjecture has been shown to be valid in many situations.
	Let $\mathcal{D}$ denote a Ryser design of order $v$, index $\lambda$ and replication numbers $r_1,r_2$. Let $e_i$ denote the number of points of $\mathcal{D}$ with replication number $r_i$ (with $i = 1, 2$). Call $A$ small (respectively large) if $|A| < 2\lambda$ (respectively $|A| > 2\lambda$) and average if $|A|=2\lambda$. Let $D$ denote the integer $e_1 - r_2$ and let $\rho> 1$ denote the rational number $\dfrac{r_1-1}{r_2-1}$. Main results of the present article are the following.\\
	For every block $A$, $r_1 \geq |A| \geq r_2$  (this improves an earlier known inequality $|A| \geq r_2$).
	If there is no small block (respectively no large block) in $\mathcal{D}$, then  $D\leq -1$ (respectively $D\geq 0$). 
	With an extra assumption $e_2 > e_1$ an earlier known upper bound on $v$ is improved from a cubic to a quadratic in $\lambda$.
	It is also proved that if $v \leq \lambda^2+ \lambda + 1$ and if $\rho$ equals $\lambda$ or $\lambda - 1$, then $\mathcal{D}$ is of Type-1.
	Finally a Ryser design with $ 2^n + 1$ points is  shown to be of Type-1.
\end{abstract}
\maketitle
\section{Introduction}

\noindent

\noindent
A design is a pair $(X,L)$, where $X$ is a finite set of points and $L\subseteq P(X)$, where $P(X)$ is the power set of $X$. The elements of $X$ are called its points and the members of $L$ are called the blocks. Most of the definitions, formulas and proofs of standard results used here can be found in \cite{Ion2}.
\newpage
\begin{defn}{\rm
		A design $\mathcal{D}=(X,L)$ is said to be a symmetric $(v,k,\lambda)$ design  if
		\begin{enumerate}[{\rm 1.}]
			\item $ |X|=|L|=v$,
			\item $ |B_1\cap B_2|=\lambda \geq 1$ for all blocks $B_1$ and $B_2$ of $\mathcal{D},~~ B_1\neq B_2$,
			\item $ |B|=k>\lambda$ for all blocks $B$ of $\mathcal{D}$.
		\end{enumerate}
}\end{defn}
\begin{defn}{\rm
		A design $\mathcal{D}=(X,L)$ is said to be a Ryser design of order $v$ and index $\lambda$ if it satisfies the following:
		\begin{enumerate}[{\rm 1.}]
			\item $|X|=|L|=v$,
			\item $ |B_1\cap B_2|=\lambda$ for all blocks $B_1$ and $B_2$ of $\mathcal{D}, B_1\neq B_2$,
			\item $ |B|>\lambda$ for all blocks $B$ of $\mathcal{D}$,
			\item there exist blocks $B_1$ and $B_2$ of $\mathcal{D}$ with $|B_1|\neq|B_2|$.
		\end{enumerate}
}\end{defn}
\noindent
Here condition 4 distinguishes a Ryser design from a symmetric design, and condition 3 disallows repeated blocks and also any block being contained in another block.\\% The term $\lambda$ is called the index of the Ryser design.\\
Woodall  \cite{wodl} introduced a new type of combinatorial object which is a combinatorial dual of a Ryser design. All known examples of Ryser designs can be described by the following construction given by Ryser which is also known as the Ryser-Woodall complementation.\\
Let $\mathcal{D}=(X,\mathcal{A})$ be a symmetric $(v,k,k-\lambda)$ design with $k\neq 2\lambda$. Let $A$ be a fixed block of $\mathcal{D}$.
Form the collection $\,\mathcal{B}=\{A\}\bigcup \{A\triangle B: B\in \mathcal{A}, B\neq A\}$, where $A\triangle B$ denotes the usual symmetric difference of $A$ and $B$. Then $\overline{\mathcal{D}}=(X,\mathcal{B})$ is a Ryser design of order $v$ and index $\lambda$ obtained from $\mathcal{D}$ by block complementation with respect to the block $A$. We denote $\overline{\mathcal{D}}$ by $\mathcal{D}*A$. %We can take points of $\mathcal{D}*A$ to be the same as those of $\mathcal{D}$.\\
Then $A$ is also a block of $\mathcal{D}*A$ and the original design $\mathcal{D}$ can be obtained by complementing $\mathcal{D}*A$ with respect to the block $A$. 
If $\mathcal{D}$ is a symmetric $(v,k,\lambda^{'})$ design, then the design obtained by complementing $\mathcal{D}$ with respect to some block is a Ryser design of order $v$ with index $\lambda=k-\lambda^{'}$. A Ryser design obtained in this way is said to be of \textbf{Type-1}.\\
Define a Ryser design to be of \textbf{Type-2} if it is not of Type-1. We now state
\begin{center}
	\textbf{The Ryser Design Conjecture \cite{Ion2}: Every Ryser design is of Type-1.}
\end{center}
\noindent
In a significant paper Singhi and Shrikhande \cite{Sin} proved the conjecture when the index $\lambda$ is a prime. In  \cite{Ser2} Seress showed the truthfulness of the conjecture for $\lambda=2p$, where $p$ is a prime. In \cite{Ion1} Ionin and Shrikhande  developed a new approach to the Ryser design conjecture that led to new results for certain parameter values. They also gave an alternate proof of the celebrated non-uniform Fisher Inequality. Ionin and Shrikhande went on to explore the validity of the Ryser design conjecture from a different perspective. Their results prove the conjecture for certain values of $v$ rather than for $\lambda$. Both Ryser and Woodall independently proved the following result:
\begin{thm}[{\cite[Theorem 14.1.2]{Ion2}}{ Ryser Woodall Theorem}]\label{thm:RyserWoodall}
	If $\mathcal{D}$ is a Ryser design of order $v$, then there exist integers $r_1$ and $r_2$, $r_1\neq r_2$ such that $r_1+r_2=v+1$ and any point occurs either in $r_1$ blocks or in $r_2$ blocks.
\end{thm}
\noindent
Let $\mathcal{D}$ be a Ryser design of order $v$ and index $\lambda$. It is known that a Ryser design has two invariants $r_1 > r_2$ with $r_1 + r_2 = v + 1$ such that every point has replication number either $r_1$ or $r_2$.  Let $g$ denote the gcd between $r_1 - 1$ and $r_2 - 1$ and let $c$ and $d$ respectively denote the integers $\dfrac{r_1 - 1}{g}$ and $\dfrac{r_2 - 1}{g}$. Let $\rho=\dfrac{r_1-1}{r_2-1}=\dfrac{c}{d}$. Then  $r_1+r_2=v+1 $ implies  $g \text{ divides } (v-1)\text{ and } v-1=(c+d)g$.
Let $a=c - d$ then any two of $c, d$ and $a$ are coprime to each other.
The point-set is partitioned into subsets $E_1$ and $E_2$, where $E_i$ is the set of points with replication number $r_i$ and let $e_i=|E_i|$ for $ i = 1, 2$. Then $e_1,e_2 > 0$ and $e_1 + e_2 = v$.
For a block $A$, let $\tau_i(A)$ denote $|E_i\cap A|$, the number of points of block $A$ with replication number $r_i$ for $i=1,2$.
Then $|A|=\tau_1(A)+\tau_2(A)$. \textit{We say a block $A$ is large, average or small if $|A|$ is greater than $2\lambda$, equal to $2\lambda$ or less than $2\lambda$ respectively.  A block which is not average is called a non average block.}  \\
The Ryser-Woodall complementation of a Ryser design $\mathcal{D}$ of index $\lambda$ with respect to some block $\,A \in \mathcal{D}\,$ is either a symmetric design or a Ryser design of index $(|A|-\lambda)$. If $\mathcal{D}*A$ is the new Ryser design of index $(|A|-\lambda)$ obtained by Ryser-Woodall complementation of a Ryser design $\mathcal{D}$ with respect to the block $A$, we denote the new parameters of $\mathcal{\mathcal{D}*A}$ by $\lambda(\mathcal{D}*A), e_1(\mathcal{D}*A)$ etc.\\
Let $\mathcal{D}_r(X)$ denote the set of all incidence structures $\mathcal{D}=(X,\mathcal{B})$ where $\mathcal{B}$ is a set of subsets of $X$ and $\mathcal{D}$ is a Ryser design with replication numbers $r_1 \text{ and } r_2=v+1-r_1$; or a symmetric design with block size $r_1 \text{ or } r_2$.
%We now quote some useful results from Ionin and Shrikhande \cite{Ion2}.
\begin{proposition}[{\cite[Proposition 14.1.7]{Ion2}}] \label{prop:complement-properties}
	Let $\mathcal{D}\in \mathcal{D}_r(X)$ and let $A,B$ be blocks of 
	$\mathcal{D}$. Then $\mathcal{D}* A\in \mathcal{D}_r(X)$ and the following 
	conditions hold: 
	\begin{enumerate}[{\rm (i)}]
		\item $(\mathcal{D}*A)*A=\mathcal{D}$;
		%if $B$ is a block o Df $b\neqA a$ then B\trangle A is a block of D*A and D*A*A\trangle B =D*B 
		\item $A\triangle B$ is a block of $\mathcal{D}*A$ and $(\mathcal{D}*A 
		)*(A\triangle B)=\mathcal{D}*(B)$;
		\item $r_1(\mathcal{D}*A)=r_1(\mathcal{D})$;
		\item $\lambda(\mathcal{D}*A)=|A|-\lambda(\mathcal{D})$; 
		\item $E_1(\mathcal{D}*A)=E_1(\mathcal{D})\triangle A$;
		\item 
		$e_1(\mathcal{D}*A)=e_1(\mathcal{D})-\tau_1(A)(\mathcal{D})+
		\tau_2(A)(\mathcal{D}
		)$;
		\item $\mathcal{D}*A$ is a symmetric design if and only if 
		$A=E_1(\mathcal{D})\text{ or } A=E_2(\mathcal{D})$. 
	\end{enumerate}
\end{proposition}
\begin{remark}\label{remark:1}
	Since $|A\triangle B|=|A|+|B|-2|A\cap B|$,  observe that if a design is of Type-1 then it has all average blocks except for one, and hence a Type-2 Ryser design must have at least two non average blocks.   
\end{remark}

\begin{thm}[{\cite[Theorem 14.1.17]{Ion2}}] 
	\label{thm:rho<lambda}
	For any Ryser design with index $\lambda>1$ and replication numbers $r_1 \text{ and } r_2$, 
	$\quad\displaystyle{\frac{\lambda}{\lambda-1}\leq \rho \leq \lambda}\text{ and } \rho \notin (\lambda-1,\lambda)$, where $\rho=\dfrac{r_1-1}{r_2-1}$.
\end{thm}    
\begin{thm}[{\cite[Theorem 1.7]{Ser1}}]\label{thm:e1e2=lambda(v-1)}
	A Ryser design is of Type-1 if and only if $\,e_1e_2=\lambda(v-1)$.
\end{thm}
\begin{thm}[{\cite[Theorem 14.4.8]{Ion2}}]\label{thm:lambda<9}
	All Ryser designs of index less than 9 are of Type-1.
\end{thm}
\begin{thm}[{\cite[Theorem 14.1.20]{Ion2}}]\label{thm:a<lambda}
	If $r_1, r_2$ are the replication numbers of a Ryser design of index $\lambda>2$, then $r_1-r_2\leq (\lambda-1)g, \text{ where } g=\gcd(r_1-1,r_2-1)$.    
\end{thm}
\begin{thm}[{\cite[Corollary 14.1.16]{Ion2}}]\label{thm:g=1isPencil}
	If $\mathcal{D}$ is a Ryser design with $g = 1$ , then $\mathcal{D}$ is of Type-1.
\end{thm}
\begin{thm}[{\cite[Theorem 14.1.19]{Ion2}}]\label{thm:upper-bound-for-v}
	The number of points of a Ryser design of index $\lambda>1$ does not exceed $\lambda^3+2$ and therefore for any fixed $\lambda>1$, there are only finitely many Ryser designs of index $\lambda$.
\end{thm}
%    \newpage
\noindent
For any symmetric $(v,k,\lambda)$ design $\mathcal{D}$, the parameter $~n=k-\lambda$ is called the order of $\mathcal{D}$. It is well known that $4n-1\leq v\leq n^2+n+1~\text{ with }~ v=4n-1$ if and only if $ \mathcal{D} ~\text{ is a Hadamard }(4n-1,2n-1,n-1)$ design or its complement, and $ v=n^2+n+1 \text{ if and only if } \mathcal{D} $ is a projective plane of order \textit{n} or its complement.
Ionin and Shrikhande \cite{Ion1} made the following conjecture.
\begin{con}\label{con:bound_on_v}
	For any Ryser design on  \textit{v}  points $\mbox{ }4\lambda-1\leq v\leq\lambda^2+\lambda+1$.
\end{con}
\noindent
In this article, we show that for any block $A$ of a Ryser design, $r_1\geq |A|\geq r_2$.  Let $D=e_1-r_2$. We establish a relation between the block size and the parameter $D$. 
We prove that if there is no small block (respectively no large block) in a Ryser design, then  $D\leq -1 \;(\text{respectively }\; D\geq 0)$.
In support of Conjecture \ref{con:bound_on_v}$~~$ we prove that $e_2>e_1$ implies $2\lambda^2+3\lambda-1>v\geq 4\lambda-1$. We also prove that Ryser designs with number of points at most $\lambda^2+\lambda +1$ and $\rho=\lambda$ or $\rho=\lambda-1$ are Type-1. Finally we prove that Ryser design with $v = 2^n + 1$ points is of Type-1.
\section{A bound on the block sizes of a Ryser design}
\noindent
%Following Seress {\cite{Ser2}} define $D = e_1 - r_2$ and following Singhi and Shrikhande {\cite{Sin}} define $\rho=\dfrac{r_1-1}{r_2-1}=\dfrac{c}{d}$, where $\gcd(c,d)=1.$ 
%Let $g=\gcd(r_1-1,r_2-1).$ 
%	Then  $r_1+r_2=v+1 $ implies  $g \text{ divides } (v-1)$,  
%$~~r_1-1=cg,~ r_2-1=dg \text{ and } v-1=(c+d)g$.
%Also write $a$ to denote $c - d$ and observe that any two of $c, d$ and $a$ are coprime to each other.
By Theorem \ref{thm:a<lambda}$~~$ we have 
\begin{equation}\label{a_less_than_lambda}
a<\lambda
\end{equation}
We use the following equations which can be found in \cite{Sin} and\cite{Ion1}.
\begin{align}
&e_1r_1(r_1 - 1) + e_2r_2(r_2 - 1) = \lambda v(v - 1) \label{e1r1r1-1+e2r2r2-1}\\
&(\rho-1)e_1=\lambda(\rho + 1)-r_2 \label{e1form}\\
&e_1=\lambda + \frac{\lambda+D}{\rho} \label{e1Dform}\\
&(\rho-1)e_2=\rho r_1-\lambda(\rho + 1) \label{e2form}\\
&e_2= \lambda +[\lambda -(D+1)]\rho \label{e2Dform}
\end{align}
\noindent
From equations (\ref{e2form}) and (\ref{e1form}), respectively:
\begin{align}
&r_1 = 2\lambda + \left(\frac{a}{c}\right)(e_2-\lambda) \label{r1form}\\
&r_2=2\lambda-\left(\frac{a}{d}\right)(e_1-\lambda) \label{r2form}
\end{align}
\noindent
%In a Ryser design a block $A$ with $|A| = 2\lambda$ is called an average block. We also define $A$ to be small (respectively large) if $|A| < 2\lambda$ (respectively if $|A| > 2\lambda$).
For a block $A$ with $|A|=\tau_1(A)+\tau_2(A)$, a simple two way counting gives 
\begin{equation}\label{tau1r1-1+tau2r2-1}
(r_1-1)\tau_1(A)+(r_2-1)\tau_2(A)=\lambda(v-1)
\end{equation}
\noindent
Dividing through by $g$, the common g.c.d. of  $r_1 -1, r_2 - 1$ and $v - 1$ yields: $ c\tau_1(A) + d\tau_2(A) = \lambda(c + d) $ and hence 
$ c( \tau_1(A) - \lambda)  + d( \tau_2(A) - \lambda) = 0. $
Using the coprimality of $c$ and $d$ then shows that $c$ divides $\tau_2(A) - \lambda$ while $d$ divides  $\tau_1(A) - \lambda$ and writing the ratios to be $t$ and $s$ respectively, it is clear that $s = - t$ and hence $\tau_1(A) - \lambda = -td$ and 
$\tau_2(A) - \lambda = tc$ for some integer $t$. 
That is
\begin{align}
&\tau_1(A)=\lambda-td\label{tau1form}\\
&\tau_2(A)=\lambda+tc\label{tau2form}\\
&|A|=2\lambda+ta\label{sizeofAform}
\end{align}
Hence we get the following lemma.
\begin{lemma}\label{lemma:blocksize}
	Let $A$ be a block of a Ryser design. Then the size of $A$ has the form  $|A|=2\lambda+ta$, where $t$ is an integer. The block $A$ is large, average or small depending on whether $t>0,  t=0$ or $t<0$ respectively. Hence $\tau_1(A)=\tau_2(A)=\lambda$ if $A$ is an average block,  $\tau_1(A)>\lambda>\tau_2(A)$ if $A$ is a small block and  $\tau_2(A)>\lambda>\tau_1(A)$ if $A$ is a large block.
\end{lemma}
\begin{thm}\label{thm:bond_on_the_blocksize}
	Every block $A$ of a Ryser design has block size bounded by $r_1\geq |A|\geq r_2$.
\end{thm}
\begin{proof}
	Since $\tau_1(A)\leq e_1$, equation (\ref{tau1form}) gives $-t \leq \dfrac{e_1 - \lambda}{d}$, and hence  we get $ 2\lambda+ta  \geq  2\lambda-(e_1 - \lambda)\dfrac{a}{d}$. By equations (\ref{sizeofAform}) and (\ref{r2form}), we get $ |A|\geq r_2$.
	Similarly, since $\tau_2(A)\leq e_2$, equation (\ref{tau2form}) gives $t \leq \dfrac{e_2 - \lambda}{c}$, and hence $2\lambda+ta  \leq  2\lambda+(e_2 - \lambda)\dfrac{a}{c}$. Using equations (\ref{sizeofAform}) and (\ref{r1form}), we get $ |A|\leq r_1$. Therefore, we have
	$r_1\geq |A|\geq r_2$. 
\end{proof}
\section{Some results on the Ryser design conjecture}
\noindent
For a Ryser design with blocks $|A_i|=k_i \text{ for } i=1,2,\ldots,v $ the column sum of the incidence matrix is equal to the row sum of the incidence matrix which implies $\sum k_i= e_1r_1+e_2r_2$.\\
Hence from equation (\ref{e1form}) and (\ref{e2form}), we get
\begin{equation}\label{e1r1e2r2form}
e_1r_1+e_2r_2=\lambda(v-1)+r_1r_2
\end{equation}
In equation (\ref{r1form}), let $x=\dfrac{e_2-\lambda}{c}$ and in equation (\ref{r2form}), let $y=\dfrac{e_1-\lambda}{d}$. 
Then, $r_1 = 2\lambda +xa$ and $r_2 = 2\lambda -ya$.  Since $c$ and $a$ are co-prime, it follows at once that $c$ divides $e_2 - \lambda$ and hence $x$ is an integer. By equation (\ref{e2Dform}) we get
\begin{equation}\label{xform}
x=\dfrac{e_2-\lambda}{c}=\dfrac{\lambda -(D+1)}{d}
\end{equation}
The assertion that $y$ is an integer follows similarly.  By equation (\ref{e1Dform}) we get
\begin{equation}\label{yform}
y=\dfrac{e_1-\lambda}{d}=\dfrac{\lambda +D}{c}
\end{equation}
From equations (\ref{xform}) and (\ref{yform}) we have:
\begin{align*}
x=\dfrac{\lambda -(D+1)}{d}&\Rightarrow xd= \lambda-(D+1) &
x=&\dfrac{e_2-\lambda}{c}\Rightarrow xc= e_2-\lambda\\
y=\dfrac{\lambda +D}{c}&\Rightarrow yc= \lambda +D &
y=&\dfrac{e_1-\lambda}{d}\Rightarrow yd=e_1-\lambda\\
\end{align*}
This gives us the following equations:
\begin{equation}\label{xcplusyd}
xc+yd=v-2\lambda,
\end{equation}
\begin{equation}\label{xdplusyc}
xd+yc=2\lambda-1
\end{equation}
On adding equations (\ref{xcplusyd}) and (\ref{xdplusyc}) we get $x(c+d)+y(c+d)=v-1$. Since $~v-1=(c+d)g~$ we have
\begin{equation}\label{xplusy}
x+y=g
\end{equation}
Subtracting equation (\ref{xdplusyc}) from equation (\ref{xcplusyd}) obtains $x(c-d)-y(c-d)=v-(4\lambda-1)$. Hence we have,
\begin{equation}\label{x-y}
x-y=\frac{v-(4\lambda-1)}{a}
\end{equation}
\begin{lemma}\label{lemma:v>=4lambda-1iffe2>e1}
	Let $\mathcal{D}$ be a Ryser design of order $v$ and index $\lambda$. Then the following conditions are equivalent.
	\begin{enumerate} [(i)]
		\item $v\geq 4\lambda-1$
		\item $x\geq y$
		\item $e_2 > e_1$. 
	\end{enumerate}
\end{lemma}
\begin{thm}
	The number of points of a Ryser design of index $\lambda>1$ does not exceed $2\lambda^2+3\lambda-1$ and therefore for any fixed $\lambda>1$, there are only finitely many Ryser designs of index $\lambda$. Further $e_2>e_1$ implies $2\lambda^2+3\lambda-1>v\geq 4\lambda-1$.
\end{thm}
\begin{proof}
	Using equation (\ref{x-y}), we have 
	$(x-y)a=v-(4\lambda-1)$ and hence $v\geq4\lambda-1$ if and only if $x\geq y$. Also, $~~v=4\lambda-1+(x-y)a~~$ and from equation (\ref{a_less_than_lambda}) we know that $a<\lambda$. Hence, $~v<4\lambda-1+(x-y)\lambda$. From equation (\ref{xdplusyc}), $0\leq x-y<xd+yc=2\lambda-1$. Therefore, $~~v<4\lambda-1+(x-y)\lambda<4\lambda-1+(2\lambda-1)\lambda$, that is, $v< 2\lambda^2+3\lambda-1$. Lemma \ref{lemma:v>=4lambda-1iffe2>e1}$~~$ completes our proof.
\end{proof}
\noindent
\begin{proposition}\label{prop:Dandblocksize}
	If there is no small block (respectively no large block) in a Ryser design, then  $D\leq -1 \;(\text{respectively }\; D\geq 0)$.    
\end{proposition}
\begin{proof}
	From equation (\ref{e1r1e2r2form}), 
	$r_1(e_1-r_2)+e_2r_2=\lambda(v-1)$. This gives $\,\,r_1D+e_2r_2=\lambda v-\lambda \text{ which implies } r_1D+\sum\limits_{A} (\tau_2(A)) = \sum\limits_{A}(\lambda)-\lambda \text{ and hence } r_1D = \sum\limits_{A}(\lambda-\tau_2(A))-\lambda$. If there is no small block in the Ryser design then all the blocks are large or average. Using Lemma \ref{lemma:blocksize},$~~ \tau_2(A)\geq\lambda$ for a large or an average block $A$ and hence if we do not have a small block then  $D\leq -1$.
	Proof in the other case is similar.
\end{proof} 
\noindent
\begin{thm}\label{thm:type-1iffD(D+1)=0}
	A Ryser design is of Type-1 if and only if $D=0 \text{ or } D=-1$. We have the following cases for different values of $D$.
	\begin{enumerate}[(i)]
		\item $D=0$ if and only if $yc=\lambda\text{ and }  xd=\lambda-1$.
		\item $D=-1$ if and only if $xd=\lambda \text{ and } yc=\lambda-1$.
		\item $D>0 \text{ if and only if } yc>\lambda >\lambda-1> xd$.
		\item $D<-1 \text{ if and only if } xd>\lambda >\lambda-1> yc$.
	\end{enumerate}
\end{thm}
\begin{proof}
	Equations (\ref{e1Dform}) and (\ref{e2Dform}) imply $e_1e_2=\lambda(v-1)-D(D+1)$.
	By Theorem \ref{thm:e1e2=lambda(v-1)} a Ryser design is of Type-1 if and only if $\,\,D(D+1)=0$. That is a Ryser design is of Type-1 if and only if $\,\,D=0\,\, \text{or}\,\, D=-1$.\\
	Equations (\ref{xform}) and (\ref{yform}) imply
	\begin{align}
	D+1&=\lambda-xd \label{D+1=lambda-xd}\\
	D&=yc-\lambda \label{D=yc-lambda}
	\end{align}
	Then all the four statements follow from the two equations above.
\end{proof}
\noindent
By Theorem \ref{thm:rho<lambda},$~~$ we have $~~\rho \leq \lambda\text{ and } \rho \notin (\lambda-1,\lambda)$. Therefore we consider the cases with $\rho=\lambda$ and $\rho=\lambda-1$.
\begin{lemma}\label{lemma:D<-1_for_rho=lambda,lambda-1}
	Let $\mathcal{D}$ be a Type-2 Ryser design with $\rho=\lambda$ or $\rho=\lambda-1$. Then $D<-1$.
\end{lemma}
\begin{proof}
	We prove the result for $\rho=\lambda-1$. The case $\rho=\lambda$ is similar. \\
	Since $\rho = \dfrac{c}{d}, \text{ with } \gcd(c,d)=1$ we have $\,c=\lambda-1, ~~~ d=1$ and hence, $a=c-d=\lambda-2$.
	By remark \ref{remark:1},$~~ \mathcal{D}$ has at least two non average blocks. Complement $\mathcal{D}$ with respect to any non average block of size $k$ to obtain a new Ryser design $\overline{\mathcal{D}}$. Let $\overline{\rho}=\dfrac{\overline{r_1}-1}{\overline{r_2}-1}$, where $\overline{r_1} \text{ and } \overline{r_2}$ are the replication numbers and let $\overline{\lambda}$ be the index of the new design. By Proposition \ref{prop:complement-properties} (iii)$~~$ and Theorem \ref{thm:rho<lambda},$~~$ we get $\overline{\lambda}=k-\lambda\geq \overline{\rho}=\rho=\lambda-1,$ for the new design which implies $k\geq 2\lambda-1$. That is, a non average blocks of $\mathcal{D}$ is either large or small of size $2\lambda-1$.\\ 
	Since a small block has size $k=2\lambda-ta$ for some positive integer $t$ we have, $~2\lambda-ta = 2\lambda-1 \text{ which implies } t=1$ and $ a=1$. That is, $~a= \lambda-2 =1\text{ and hence } \lambda=3$. By Theorem \ref{thm:lambda<9},$~~$ $\mathcal{D}$ is of Type-1, a contradiction.\\
	Therefore, a non average block of $\mathcal{D}$ is large
	and by Proposition \ref{prop:Dandblocksize}, we have $D\leq -1$. Using the fact that $\mathcal{D}$ is of Type-2 Theorem \ref{thm:type-1iffD(D+1)=0}$~~$ now implies that $D < - 1$ completing the proof.
\end{proof}
\begin{thm}
	Let $\mathcal{D}$ be a Ryser design with at most $\lambda^2+\lambda+1$ points with $\rho=\lambda$ or $\rho=\lambda-1$. Then, $\mathcal{D}$ is of Type-1.
\end{thm}
\begin{proof}
	Suppose not. Then $\mathcal{D}$ is of Type-2. We prove that this leads to a contradiction.\\
	Let $\rho= \lambda$. Then $\lambda = \dfrac{c}{d}, \text{ with } \gcd(c,d)=1$ implies $\,c=\lambda, ~~~ d=1$ and hence, $a=c-d=\lambda-1$. By lemma (\ref{lemma:D<-1_for_rho=lambda,lambda-1}) we have $D<-1$. Equations (\ref{xcplusyd}) and (\ref{xdplusyc}) imply $x\lambda+y=v-2\lambda$ and $x+y\lambda=2\lambda-1$. Adding the two we get $(x+y)(\lambda+1)=v-1$. Then, by equation (\ref{xplusy}) we have $g=\dfrac{v-1}{\lambda+1}$. Now, $g=x+y=x+\lambda y-y(\lambda-1)$. Then, by equation (\ref{xdplusyc}) we get $\dfrac{v-1}{\lambda+1}=2\lambda-1-y(\lambda-1)$. 
	This gives us $\lambda^2+\lambda+1-v=(\lambda^2-1)(y-1)$, and hence $ y\geq1$. Since $D < - 1$, Theorem \ref{thm:type-1iffD(D+1)=0} (iv)$~~$ implies $x>\lambda >\lambda-1> y(\lambda-1)~$, with $y\geq 1$ which is impossible. The case $\rho= \lambda-1$ is similar. 
\end{proof}
\noindent 
\begin{thm}
	A Ryser design with $v = 2^n + 1$ points is of Type-1.
\end{thm}
\begin{proof}
	Note right at the beginning that the case $r_2 = 1$ leads to $r_1 = v$ forcing repeated blocks that are disallowed by 
	the stipulations on a Ryser design. Equation (\ref{e1r1r1-1+e2r2r2-1}) and division by $g$, then obtains:
	\begin{equation}\label{eq1}
	e_1(gc + 1)c + e_2(gd + 1)d = \lambda [g(c + d) + 1](c + d)
	\end{equation}
	By Theorem \ref{thm:g=1isPencil}$~~$ we are done if $g = 1$ and hence we must assume that $g \neq 1$. Since $g$ must divide $v - 1 = 2^n$ and $g \neq v - 1$, we have: $g = 2^m$ where $1 \leq m \leq n - 1$ is an integer. Let $k = m - n$. Then equation (\ref{eq1}) reduces to: 
	\begin{equation}\label{eq2}
	e_1(2^mc + 1)c + e_2(2^md + 1)d = \lambda [2^n + 1]2^k
	\end{equation}
	Here, $k$ is a positive integer showing that the right side is an even number. Since $c + d = 2^k$, it follows that $c, d$ have the same parity and since their g.c.d. is $1$, $c$ and $d$ must be both odd. On the other hand, $e_1 + e_2 = 2^n + 1$. So exactly one of $e_1$ and $e_2$ is odd say $e_1$ is odd and $e_2$ is even. On the 
	left side of equation (\ref{eq2}), the first summand is odd while the second is even forcing the right side to be odd, which is impossible.
\end{proof}

\end{document}